\documentclass[11pt]{article}
\usepackage{amssymb,amsbsy,latexsym,amsmath,bbm,epsfig,psfrag,amsthm,mathrsfs}
\usepackage[swedish,english]{babel}
\usepackage{graphicx,tikz,esint}
\usepackage[T1]{fontenc}
\usepackage[latin1]{inputenc}
\usepackage{amsfonts}
\usepackage{amsmath}
\usepackage{amssymb}
\usepackage{epsf}
\usepackage{graphics}
\usepackage{graphicx}
\usepackage{latexsym}
\usepackage{psfrag}
\usepackage{relsize}
\usepackage{verbatim}
\usepackage{hyperref}
\usepackage{pgfplots}
\usepackage{float}
\usetikzlibrary{positioning}
\usetikzlibrary{arrows}
\usetikzlibrary{decorations.pathreplacing}
\usepackage[all]{xy}
\usepackage{url}

\theoremstyle{plain}

\numberwithin{Lem}{section}
\newtheorem{Prop}{Proposition}
\numberwithin{Prop}{section}
\newtheorem{Thm}{Theorem}
\numberwithin{Thm}{section}

\numberwithin{Cor}{section}

\numberwithin{Con}{section}

\theoremstyle{definition}
\newtheorem{Def}{Definition}
\numberwithin{Def}{section}

\numberwithin{hyp}{section}

\numberwithin{conj}{section}

\numberwithin{ex}{section}

\theoremstyle{remark}
\newtheorem{rem}{\bf{Remark}}
\numberwithin{rem}{section}

\numberwithin{equation}{section}

\newcommand{\dv}{\partial}
\newcommand{\Om}{\Omega}

\newcommand{\R}{{\mathbb R}}
\newcommand{\C}{{\mathbb C}}

\newcommand{\Cc}{\mathcal{C}}

\newcommand{\LL}{\mathcal{L}}
\newcommand{\BB}{\mathcal{S}}

\newcommand{\Dp}{\mathcal{D}^{\tiny+}}
\newcommand{\Dm}{\mathcal{D}^{\tiny-}}

\begin{document}

\vspace{.4cm}

\title{\sffamily A generalized Montel theorem for a class of first order elliptic equations with measurable coefficients}
\author{Erik Duse}
\date{}
\maketitle

\begin{abstract}
In this paper we prove a generalization of Montel's theorem for a class of first order elliptic equations with measurable coefficients involving Hodge--Dirac operators. We then apply this result to sequences of solutions of second order uniformly elliptic equations with measurable coefficients on divergence form and show that this results in a precompactness result for such sequences. 
\end{abstract}

\section*{\sffamily Introduction}

We begin by recalling a version of Montel's theorem from classical complex analysis.
\begin{Thm}
Let $\Om \subset \C$ be a domain and let $\mathscr{F}$ be a family of bounded holomorphic functions on $\Om$. The $\mathscr{F}$ contains a subsequence $\{f_n\}\subset \mathscr{F}$ that converge compactly to some $f\in \mathscr{F}$.
\end{Thm}
Here converge compactly means that $f_n\to f$ uniformly on any compact subset $K\subset \Om$ as $n\to \infty$. Holomorphic functions are of course defined by the condition that 
\begin{align*}
\overline{\dv}f(z)=0,
\end{align*}
and $\overline{\dv}$ is the \emph{elliptic} Cauchy-Riemann operator. Montel's theorem can be generalized to solutions of a large class of partial differential operators with smooth coefficients acting on sections over manifolds. For simplicity we assume that we are on euclidean domains and we have vector valued fields rather than sections. Following \cite{B} we define:
\begin{Def}[Weak hypoellipticity]
Let $\Om\subset \R^n$ be a domain and let $E$ and $F$ be euclidean vector spaces and let $\mathcal{P}$ be a partial differential operator with smooth coefficients such that 
\begin{align*}
\mathcal{P}: C^\infty(\Om,E)\to C^\infty(\Om,F).
\end{align*}
We say that $\mathcal{P}$ is \emph{weakly} hypoelliptic if all solutions of $\mathcal{P}u=0$ are smooth. Furthermore, we set $\mathscr{H}(\Om,\mathcal{P})=\text{ker}(\mathcal{P})=\{u\in C^\infty(\Om,E): \mathcal{P}u=0\}$. 
\end{Def}

In \cite{B} the following Theorem 4 was proven.

\begin{Thm}\label{thm:WHypo}
Let $\mathcal{P}$ be a weakly hypoelliptic operator. Then any locally $L^1$-bounded sequence $\{u_j\}_j\in \mathscr{H}(\Om,\mathcal{P})$ contains a subsequence which converge in the $C^\infty$-topology to some $u\in \mathscr{H}(\Om,\mathcal{P})$.  
\end{Thm}

The proof of Theorem 4 relies on the estimate 
\begin{align*}
\Vert u\Vert_{C^j(K)}\leq C\Vert u\Vert_{L^1(U)}
\end{align*}
valid for weakly hypoelliptic operators $\mathcal{P}$ and any $u\in \mathscr{H}(\Om,\mathcal{P})$ and any compact set $K\subset \Om$ and open set $U\Subset \Om$ containing $K$, see Lemma 2 in \cite{B}. 

In another direction, let $\Om \subset \C$ and consider the Beltrami equation 
\begin{align}\label{eq:B}
\overline{\dv}f(z)=\mu(z)\dv f(z)
\end{align}
where $\mu\in L^\infty(\Om,\C)$ such that $\vert \mu\vert_\infty<1$. This is an elliptic system with non-smooth coefficients for which the notion of hypoellipticity is meaningless. Yet, all solution $f\in W^{1,2}_{loc}(\Om,\C)$ of the Beltrami equation admits a factorization  
\begin{align*}
f(z)=h(g(z))
\end{align*}
where $h$ is any holomorphic function and $g\in W^{1,1}_{loc}(\Om)$ is a \emph{fixed homeomorphic solution} to \eqref{eq:B}. This is the famous Stoilow factorization theorem, see \cite[Thm. 5.5.1]{AIM}. Using the Stoilow factorization theorem one can prove a generalization of Montel's theorem for solutions of the Beltrami equation by reducing it to Montel's theorem for holomorphic functions. Namely, if $\{f_j\}_j\subset W^{1,2}_{loc}(\Om,\C)$ is a family of solutions of the Beltrami equation \eqref{eq:B} that is locally bounded in $L^1$, then there exists a fixed homeomorphic solution $g\in W^{1,1}_{loc}(\Om)$ of \eqref{eq:B} and a sequence of holomorphic functions $\{h_j\}_j$ locally bounded in $L^1$ such that $f_j=h_j\circ g$. By Montel's theorem for holomorphic functions there exists a subsequence $\{h_{j_k}\}_k$ that converge locally uniformly to a holomorphic function $h$. But then $f=h\circ g$ solves \eqref{eq:B}. Our aim now is to generalize this result to a class of first order elliptic systems in higher dimensions with merely measurable coefficients. This result will require slightly stronger assumptions and the result is also slightly less general in the sense that our limit function may depend on the choice of open subset $U\Subset \Om$ considered. This is proven in Theorem \ref{thm:main}. A consequence of this result is a precompactness result for solutions of second order uniformly elliptic equations on divergence form.

\subsection*{Acknowledgements}
Erik Duse was supported by the Knut and Alice Wallenberg Foundation KAW grant 2016.0416. The author thanks Kari Astala for useful suggestions improving the presentation of the paper.

\section*{\sffamily Dirac--Beltrami equation}
 
Let $\Lambda \R^n$ denote the exterior algebra equipped with the induced euclidean inner product form $\R^n$. Let $d$ denote the exterior derivative and $\delta=-d^\ast$ the interior derivative, the negative of the formal adjoint of $d$. Define the Hodge-Dirac operators
\begin{align*}
\mathcal{D}^\pm=d\pm \delta. 
\end{align*}
Then 
\begin{align*}
(\Dp)^2=-(\Dm)^2=\Delta, \quad \Dm \Dp=-\Dp \Dm,
\end{align*}
and $\Dp$ is formally skew-adjoint and $\Dm$ is formally self-adjoint. For a more comprehensive discussion of the Hodge--Dirac operators and their relation to Clifford algebras we refer the reader to \cite{D}. Let $\Om$ be a bounded domain in $\R^n$ and $\mathcal{M}\in L^\infty(\Om, \LL(\Lambda\R^n))$ satisfy 
\begin{align*}
M=\Vert \Vert \mathcal{M}(x)\Vert\Vert_{L^\infty}<1.
\end{align*}
Consider the \emph{Dirac--Beltrami equation}
\begin{align}\label{eq:DB}
\Dm F(x)=\mathcal{M}(x)\Dp F(x)
\end{align}
for $F\in W^{1,2}_{loc}(\Om,\Lambda):=W^{1,2}_{loc}(\Om,\Lambda \R^n)$. The condition that $M<1$ ensures that this system is elliptic. This can be seen as a generalization of the classical Beltrami equation in the plane and solutions of it has many features in common with solutions to the Beltrami equation, especially in the case when $F: \Om \to \Lambda^{ev}\R^n=\oplus_{j=0}\Lambda^{2j}\R^n$. In the same time, it is also a very different equation in the sense that $\text{dim}(\Lambda \R^n)=2^n>n=\text{dim}(\R^n)$. In particular, there is no Stoilow factorization theorem for solutions of \eqref{eq:DB}. This fact necessitates the need to develop new methods to study solutions and prove analogous theorems in higher dimension. In \cite{D} the following representation formula for solutions of \eqref{eq:DB} was proven for $F\in W^{1,2}(\Om,\Lambda)$ and where $\Om$ is a $C^2$-domain.
\begin{align*}
F(x)=\mathcal{C}_T^+\circ (I-\BB_T\mathcal{M})^{-1}\BB_T\mathcal{M}(\Dp H)+H,
\end{align*}
where $H$ is a solution to the boundary value problem 
\begin{align}\label{eq:Holo}
\Dm H(x)=0\,\,\, x\in \Om, \,\,\, \gamma_TH=\gamma_TF.
\end{align}
Here $\gamma_T$ denotes the \emph{tangential trace} and $\Cc_T^+$ is the \emph{tangential Cauchy transform} and  $\BB_T$ is the \emph{tangential Beurling transform} with respect to $\Om$, provided $\Om$ is at least a Lipschitz domain. We will omit the definitions of these integral operators but instead the refer the reader to \cite{D}. In fact the representation formula also holds for strongly Lipschitz domains but then $F$ belongs to the \emph{partial Sobolev space}
\begin{align}\label{eq:Holo}
W^{1,2}_{d,\delta}(\Om,\Lambda)=\{f\in \mathscr{D}'(\Om,\Lambda): F,dF,\delta F\in L^2(\Om,\Lambda)\}. 
\end{align}
instead.

The important point is that $\Cc_T^+$ is a weakly singular integral operator and $\BB_T$ is a singular integral operator. The condition $M<1$ now enters and ensures that the Neumann series 
\begin{align*}
(I-\BB_T\mathcal{M})^{-1}G=\sum_{j=0}^\infty\BB_T\mathcal{M}G
\end{align*} 
converges for any $G\in L^2(\Om,\Lambda)$. 
If we define the operator 
\begin{align*}
\mathcal{T}=\mathcal{C}_T^+\circ (I-\BB_T\mathcal{M})^{-1}\BB_T \mathcal{M}\circ \Dp 
\end{align*} 
the representation formula can be written 
 \begin{equation}\label{eq:StowG}
 \boxed{ F=(I+\mathcal{T})(H).}
  \end{equation}

We note that $H$ is unique up to an element in the relative Hodge-De Rham space $\mathcal{H}_T(\Om)=\{\omega\in W^{1,2}(\Om,\Lambda): d\omega=\delta\omega=0,\,\,\, \gamma_T\omega=0\}$. Note that whenever $\Omega$ is finitely connected we have $\text{dim}(\mathcal{H}_T(\Om))<+\infty$. Moreover, $\Dm$ can be viewed as a generalization of $\overline{\dv}$, and solutions of $\Dm H=0$ have many properties in common with holomorphic functions. Just as the Stoilow theorem, this sets a up a bijective (up to $\mathcal{H}_T(\Om)$) correspondence between solutions of \eqref{eq:DB} and solutions of $\Dm H=0$. In particular any $H\in W^{1,2}(\Om,\Lambda)$ solving $\Dm H=0$ give rise to a solution $F$ of \eqref{eq:DB} through $F=(I+\mathcal{T})(H)$. Moreover $\mathcal{T}:W^{1,2}(\Om,\Lambda)\to W^{1,2}(\Om,\Lambda)$ is a bounded operator, which if $M<<1$ can be viewed as a small perturbation of $H$ in the $W^{1,2}$-topology.  The purpose of the present work is to exploit this correspondence and show how properties of $H\in \mathscr{H}(\Om, \Dm)$ can be transfered to properties of $F$.

We first the recall the following Caccioppoli inequality that is proven in \cite{D}. Let $\Om \subset \R^n$ be a domain and let $F\in W^{1,2}_{loc}(\Om)$ be a solution of $\Dm F=\mathcal{M}\Dp F$ on $\Om$. Then for any $\varphi\in C^\infty_0(\Om)$ there exists a constant $C$ depending only on $n$ such that
\begin{align}\label{eq:Cap}
\int_{\R^n}\vert \varphi(x)\vert^{2}\vert \nabla \otimes F(x)\vert^{2}dx\leq C\int_{\R^n}\vert \nabla \varphi(x)\vert^{2}\vert F(x)\vert^{2}dx.
\end{align}
Let $U\Subset V\Subset \Om$ be open sets such that $\text{supp}(\varphi)\subset \overline{V}$ and $\varphi$ is chosen so that $\varphi(x)=1$ for all $x\in \overline{U}$. Then \eqref{eq:Cap} shows that 
there exists a constant $C'$ depending on $n$, $U$ and $V$ so that 
\begin{align}\label{eq:Cap2}
\Vert F\Vert_{W^{1,2}(U,\Lambda)}\leq C'\Vert F\Vert_{L^{2}(V,\Lambda)}. 
\end{align}

\begin{Prop}\label{prop:Mont}
Let $\Om \subset \R^n$ be a bounded $C^2$-domain. Let $\{F_j\}_j\subset W^{1,2}(\Om,\Lambda)$ be a sequence of solutions of 
\begin{align*}
\Dm F_j(x)=\mathcal{M}(x)\Dp F_j(x)
\end{align*}
for some $\mathcal{M}\in L^\infty(\Om,\LL(\R^n))$ such that $M:=\Vert \Vert \mathcal{M}(x)\Vert\Vert_\infty<1$. If
\begin{align*}
\sup_j\Vert F_j\Vert_{W^{1,2}}<+\infty,
\end{align*}
then there exists a subsequence $\{F_{j_k}\}_k\in W^{1,2}(\Om,\Lambda)$ that converges to some $F\in  W^{1,2}(\Om,\Lambda)$ and solves 
\begin{align*}
\Dm F(x)=\mathcal{M}(x)\Dp F(x). 
\end{align*}
\end{Prop}

\begin{proof}
Since $\Om$ is a $C^2$-domain the trace map $\gamma: W^{1,2}(\Om,\Lambda)\to W^{1/2,2}(\Om,\Lambda)$ is bounded. For each $j$ define $H_j\in W^{1,2}(\Om,\Lambda)$ through \eqref{eq:Holo}. In particular we have 
\begin{align*}
\Vert H_j\Vert_{W^{1,2}(\Om)}\leq C'\Vert \gamma H_j\Vert_{W^{1/2,2}(\dv \Om)}=C'\Vert \gamma F_j\Vert_{W^{1/2,2}(\dv \Om)}\leq C''\Vert  F_j\Vert_{W^{1,2}(\Om)}.
\end{align*}
Thus 
\begin{align*}
\sup_{j}\Vert H_j\Vert_{W^{1,2}(\Om)}<+\infty. 
\end{align*}
But this implies that $\sup_j\Vert H_j\Vert_{L^1(\Om)}<+\infty$. Since the Hodge--Dirac operator $\Dm$ is elliptic, it is in particular weakly hypoelliptic. By Theorem \ref{thm:WHypo}, $\{H_j\}_j$ contains a weakly converging subsequence $\{H_k\}_k$ that converge locally uniformly in $C^\infty$ to some $H\in C^\infty(\Om,\Lambda)$ which solves $\Dm H=0$ in $\Om$. But this implies that $H_k\to H$ in $W^{1,2}(\Om,\Lambda)$. Define 
\begin{align*}
F:=(I+\mathcal{T})(H). 
\end{align*}
Then $F$ solves \eqref{eq:DB} and satisfy $\gamma_TF=\gamma_TH$. 
\end{proof}

\begin{Thm}\label{thm:main}
Let $\Om \subset \R^n$ be a bounded domain. Let $\{F_n\}_n\subset W^{1,2}_{loc}(\Om,\Lambda)$ be a sequence of solutions of \eqref{eq:DB} for some $\mathcal{M}\in L^\infty(\Om,\LL(\R^n))$ such that $M:=\Vert \Vert \mathcal{M}(x)\Vert\Vert_\infty<1$. Assume that $\{F_j\}_j$ is locally bounded in $L^2$. Then for any smooth relatively compact subset $U\Subset \Om$ there exists a subsequence $\{F_{n_k}\}_k$ that converge in $W^{1,2}$ to some $F=F_U\in W^{1,2}(\Om)$ that solves \eqref{eq:DB}.
\end{Thm}

\begin{proof}
By the Caccioppoli inequality \eqref{eq:Cap2}, for any $U\Subset \Om$ there exists an open set $V$ such that $U\Subset V\Subset \Om$ such that any solution $F\in W^{1,2}_{loc}(\Om)$ of \eqref{eq:DB} satisfy 
\begin{align*}
\Vert F\Vert_{W^{1,2}(U)}\leq \Vert F\Vert_{L^2(V)}. 
\end{align*}
Hence 
\begin{align*}
\sup_n\Vert F_n\Vert_{W^{1,2}(U)}\leq +\infty.
\end{align*}
By Proposition \ref{prop:Mont} there exists a subsequence $\{F_{n_j}\}_j$ that converge in $W^{1,2}$ to some $F_U\in W^{1,2}(U,\Lambda)$ and solves \eqref{eq:DB}. 
\end{proof}

\begin{rem}
It is not known to the author whether solutions to \eqref{eq:DB} are Hölder continuous or not. In a special case this is proven in Theorem 11.4 in \cite{D}. 
\end{rem}

\section*{\sffamily Applications}

\begin{Prop}\label{thm:main2}
Let $\Om\subset \R^n$ be a domain and assume that $\{u_n\}_n\in W^{1,2}_{loc}(\Om)$ is locally uniformly bounded in $L^2$ and solve \eqref{eq:Div} 
\begin{align}\label{eq:Div}
\text{div}\, A(x)\nabla u_j(x)=0
\end{align}
on $\Om$ for every $j$. Here $A\in L^\infty(\Om, \LL(\R^n))$ is normal, i.e., $[A^\ast(x),A(x)]=0$ and satisfy the ellipticity bounds 
\begin{align*}
\lambda \vert v\vert^2\leq \langle A(x)v,v\rangle, \quad \Vert A(x)\Vert<\Lambda 
\end{align*}      
for a.e. $x\in \Om$ and some $0<\lambda\leq \Lambda<\infty$. Then on any smooth simply connected $U\Subset \Om$ there exists a subsequence $\{u_{n_k}\}_k$ that converge to some $u\in W^{1,2}(U)$ and solves \eqref{eq:Div} on $U$. 
\end{Prop}

\begin{proof}
Let 
\begin{align*}
\mathcal{M}(x)=(I-A(x))\circ (I+A(x))^{-1}
\end{align*}
denote the Cayley transform of $A$. By Lemma A.3 in \cite{D} there exists an $M=M(\lambda,\Lambda)<1$ such that 
\begin{align*}
M=\Vert \Vert \mathcal{M}(x)\Vert\Vert_{L^\infty}.
\end{align*}
Moreover, on any smooth simply connected open $U\Subset \Om$ the equations \eqref{eq:Div} are equivalent to the equations 
\begin{align*}
\Dm F_j(x)=\mathcal{M}(x)\Dp F_j(x)
\end{align*}
for some $F_j=u_j+v_j$, where $v_j\in W^{1,2}(\Om, \Lambda^2\R^n)$. Applying Theorem \ref{thm:main} to the sequence $\{F_j\}$ and using Theorem 6.2 in \cite{D} proves the desired result.  
\end{proof}

{\sc Erik Duse}, Department of Mathematics and Statistics, KTH,  Stockholm, Sweden \texttt{duse@kth.se}

\end{document}